\documentclass[reqno,12pt]{amsart}
\usepackage{amssymb}
\usepackage{mathrsfs}
\usepackage{txfonts}
\usepackage{amsfonts}
\usepackage{amstext}
\usepackage{amssymb}
\usepackage{amsmath}
\usepackage{graphicx}
\usepackage{hyperref}
\usepackage{url}
\usepackage{amssymb}
\usepackage{subfig}

\hypersetup{
        colorlinks   = true,
        citecolor    = blue,
        linkcolor    = blue,
        urlcolor     = blue
}
\allowdisplaybreaks
\newtheorem{theorem}{Theorem}[section]
\newtheorem{proposition}[theorem]{Proposition}
\newtheorem{lemma}[theorem]{Lemma}

\newtheorem{remark}[theorem]{Remark}

\numberwithin{equation}{section}
\newcounter{counterConstant}

\newcommand{\dimH}{\dim_{\mathrm H}}
\newcommand{\ldimB}{\underline{\dim}_{\mathrm B}}
\newcommand{\udimB}{\overline{\dim}_{\mathrm B}}
\newcommand{\norm}[1]{\left\lVert #1\right\rVert}
\newcommand{\Frob}{\mathrm F}

\begin{document}
\title[]{Sub-self-similar sets with distinct Hausdorff and Box dimensions}
\author[Zhang]{Junda Zhang}
\address{School of Mathematics, South China University of Technology,
Guangzhou 510641, China.}
\email{summerfish@scut.edu.cn}

\begin{abstract}
We construct an IFS consisting of three contracting similarities on $\mathbb{R}^{9}$ with a
countable compact sub-self-similar set $E$ satisfying
\[
\dimH E=0
<
\frac{\log 2}{\log 10}
\leq \ldimB E.
\]
This provides a concrete counterexample to the question raised by
Falconer concerning the equality of Hausdorff and box dimensions for
sub-self-similar sets without the open set condition.
\end{abstract}

\date{\today}
\maketitle
\tableofcontents

\section{Introduction}

Let $S_1,\dots,S_m$ be contracting similarities of $\mathbb{R}^d$.
Following Falconer \cite{Falconer1995}, a nonempty compact set
$E\subset\mathbb{R}^d$ is called \emph{sub-self-similar} with respect
to the IFS $\{S_1,\dots,S_m\}$ if
\[
E\subseteq \bigcup_{i=1}^{m}S_i(E).
\]

For a nonempty bounded set $F\subset\mathbb{R}^d$, denote by $N_\delta(F)$
the minimum number of sets of diameter at most $\delta$ required
to cover $F$. Its lower and upper box dimensions are
\[
\ldimB F
=
\liminf_{\delta\downarrow0}
\frac{\log N_\delta(F)}{-\log\delta}
\]
and
\[
\udimB F
=
\limsup_{\delta\downarrow0}
\frac{\log N_\delta(F)}{-\log\delta},
\]
respectively.

Falconer \cite{Falconer1995} proved that if the corresponding IFS satisfy the open set condition, then the Hausdorff, lower box, and
upper box dimensions of the sub-self-similar set coincide. He asked whether these dimensions still coincide when the open set condition is removed, and conjectured that the answer should be negative.
This paper confirms his conjecture and provide a concrete countable set as a counterexample in $\mathbb{R}^{9}.$

\begin{theorem}\label{thm:main}
There exist an IFS consisting of three contracting similarities
\[
S_0,S_A,S_B:\mathbb{R}^{9}\longrightarrow\mathbb{R}^{9}
\]
and a countable compact set $E\subset\mathbb{R}^{9}$ such that
\[
E\subseteq S_0(E)\cup S_A(E)\cup S_B(E)
\]
with
\[
0=\dimH E
<
\frac{\log 2}{\log 10}
\leq \ldimB E.
\]
\end{theorem}

We will view a $3\times 3$ matrix $A$ (equipped with the Frobenius norm ) as a vector in $\mathbb{R}^{9}$(equipped with the Euclidean metric). The construction uses two rotations that
belong to a classical class of free rotation pairs, see
\cite{Swierczkowski1994}. The organization of this paper is as follows. We first present the precise construction. Then we calculate the lower bound of the lower box dimension with the help of an estimate on the matrix product. Finally, we end with some discussions.

\section{The construction}

Consider the matrices
\[
A=
\begin{pmatrix}
1&0&0\\[2mm]
0&\frac35&-\frac45\\[2mm]
0&\frac45&\frac35
\end{pmatrix},
\qquad
B=
\begin{pmatrix}
\frac35&-\frac45&0\\[2mm]
\frac45&\frac35&0\\[2mm]
0&0&1
\end{pmatrix}.
\]
Thus $A$ and $B$ are rotations through the angle $\theta$ satisfying
\[
\cos\theta=\frac35,
\qquad
\sin\theta=\frac45,
\]
about the $x$-axis and the $z$-axis, respectively. In particular,
$A,B\in \operatorname{SO}(3).$

For a finite word
\[
w=i_1\cdots i_n\in\{A,B\}^n
\]
with letter $A$ or $B$, we write the corresponding matrix (according to the same order) by
\[
R_w=i_1\cdots i_n.
\]
For the empty word $\varnothing$, let $R_\varnothing=I_3$. Define
\[
\mathcal{W}_n
=
\{R_w:w\in\{A,B\}^n\},
\qquad
\mathcal{W}_0=\{I_3\}.
\]

Let
\[
\mathbb{M}=M_3(\mathbb{R}),
\]
identified with $\mathbb{R}^{9}$ and equipped with the Frobenius inner
product
\[
\langle X,Y\rangle_{\Frob}
=
\operatorname{tr}(X^{\mathsf T}Y)
\]
and Frobenius norm
\[
\norm{X}_{\Frob}
=
\bigl(\operatorname{tr}(X^{\mathsf T}X)\bigr)^{1/2}.
\]

If $R\in\operatorname{SO}(3)$, then clearly
\[
\langle RX,RY\rangle_{\Frob}
=
\operatorname{tr}(X^{\mathsf T}R^{\mathsf T}RY)
=
\operatorname{tr}(X^{\mathsf T}Y)
=
\langle X,Y\rangle_{\Frob}.
\]

Define
\[
S_A(X)=\frac12AX,
\qquad
S_B(X)=\frac12BX,
\qquad
S_0(X)=I_3+\frac12X.
\]
Each of these maps is a contracting similarity of $\mathbb{M}$ with contraction ratio $1/2$.

For $n\geq0$, let
\[
E_n=2^{-n}\mathcal{W}_n
=
\{2^{-n}R:R\in\mathcal{W}_n\},
\]
and our construction is given by
\[
E=\{0\}\bigcup\cup_{n=0}^{\infty}E_n.
\]

\begin{proposition}\label{prop:compact}
The set $E$ is countable and compact.
\end{proposition}

\begin{proof}
Clearly, since each set $E_n$ is finite, $E$ is countable. To see the compactness, note that each set $E_n$ is located on a sphere of $\sqrt{3}\,2^{-n}$ since when $R\in\operatorname{SO}(3)$,
\[
\norm{R}_{\Frob}^2
=
\operatorname{tr}(R^{\mathsf T}R)
=
\operatorname{tr}(I_3)
=
3.
\]
Thus, every sequence in $E$ has a subsequence either converging to 0 as
\[
\sup_{X\in E_n}\norm{X}_{\Frob}
=
\sqrt{3}\,2^{-n}
\longrightarrow0,
\]
or all elements are uniformly bounded away from 0, which implies the existence of a constant subsequence. Since $\mathbb{M}\cong\mathbb{R}^{9}$ is a metric space, $E$ is compact.
\end{proof}

\begin{proposition}\label{prop:sss}
The set $E$ is sub-self-similar with respect to
$\{S_0,S_A,S_B\}$,namely,
\[
E\subseteq S_0(E)\cup S_A(E)\cup S_B(E).
\]
\end{proposition}

\begin{proof}
Note that
\[
0=S_A(0)
\]
and
\[
I_3=S_0(0),
\]
so both $0$ and the unique point in $E_0$ belong to the right-hand side.
It remains to check  $E_n$. Let $n\geq1$ and let
\[
R_w=i_1i_2\cdots i_n\in\mathcal{W}_n,
\qquad
i_j\in\{A,B\}.
\]
We have
\[
2^{-n}R_w
=
S_{i_1}
\left(
2^{-(n-1)}i_2\cdots i_n
\right).
\]
Therefore each
point in $E_n$ belongs to either $S_A(E)$ or $S_B(E)$, showing the
desired.
\end{proof}

\section{Proof of Theorem~\ref{thm:main}}

We need the following lemma to estimate the lower box dimension.
\begin{lemma}\label{lem:free-semigroup}
For every $n\geq1$, the map
\[
w\longmapsto R_w,
\qquad
\{A,B\}^n\longrightarrow \operatorname{SO}(3),
\]
is injective. Consequently,
\[
\#\mathcal{W}_n=2^n.
\]
Moreover, if $u,v\in\{A,B\}^n$ are distinct, then
\[
\norm{R_u-R_v}_{\Frob}\geq 5^{-n}.
\]
\end{lemma}
     We will use the following Świerczkowski's Free Rotation Theorem in  \cite{Swierczkowski1994}. Recall that a \emph{rotation} in $SO(3)$ is an element whose matrix has trace $1+2\cos\theta$ for some $\theta\in[0,\pi]$; the real number $\theta$ is its \emph{angle of rotation}, and the line through the origin fixed by the rotation is its \emph{axis}.
\begin{theorem}[Świerczkowski, 1958]
\label{thm:swierczkowski}
Let $R_{1}, R_{2}\in SO(3)$ be two rotations with the same angle $\theta\in(0,\pi)$ and with mutually perpendicular axes. Suppose that $\cos\theta\in\mathbb{Q}$. Then the subgroup
\[
\langle R_{1}, R_{2} \rangle := \{ \, w(R_{1},R_{2}) \mid w \text{ is a word in the free group } F_{2} \,\}
\]
is a free group of rank $2$ if and only if
\[
\cos\theta \notin \left\{ 0,\; \pm\frac{1}{2},\; \pm 1 \right\}.
\]
\end{theorem}

\begin{proof}
By  Świerczkowski's Free Rotation Theorem, the
map $w\mapsto R_w$ is injective and
\[
\#\mathcal{W}_n=2^n.
\]
Note that every entry of $R_w$ belongs to $5^{-n}\mathbb{Z}$. If
$u\neq v$, then $R_u-R_v$ is a nonzero matrix whose entries belong to
$5^{-n}\mathbb{Z}$, and hence
\[
\norm{R_u-R_v}_{\Frob}\geq 5^{-n}.
\]
\end{proof}
We are now in a position to prove Theorem~\ref{thm:main}.

\begin{proof}
By Propositions~\ref{prop:compact} and
\ref{prop:sss}, it remains  to estimate  the dimensions. Since $E$ is countable,
\[
\dimH E=0.
\]
By
Lemma~\ref{lem:free-semigroup}, the set $E_n$ contains exactly $2^n$
points. If $u,v\in\{A,B\}^n$ are distinct, then
\[
\begin{aligned}
\norm{2^{-n}R_u-2^{-n}R_v}_{\Frob}=
2^{-n}\norm{R_u-R_v}_{\Frob}\geq
2^{-n}5^{-n}=
10^{-n}.
\end{aligned}
\]
Thus the points of $E_n$ are pairwise $10^{-n}$-separated.

Let
\[
\delta_n=\frac12\,10^{-n}.
\]
Any set of diameter at most $\delta_n$ contains at most one point of
$E_n$, showing
\[
N_{\delta_n}(E)\geq2^n.
\]
Thus when
$
\delta_{n+1}<\delta\leq\delta_n,
$
we have
\[
N_\delta(E)\geq2^n.
\]
It follows that
\[
\frac{\log N_\delta(E)}{-\log\delta}
\geq
\frac{n\log2}
{(n+1)\log10+\log2}.
\]
Letting $\delta\downarrow0$, and hence $n\to\infty$, gives
\[
\ldimB E
\geq
\frac{\log2}{\log10}.
\] The proof is now complete.
\end{proof}

\section{Further discussions}

Recall that an iterated function system
$\{T_1,\dots,T_m\}$ satisfies the \emph{open set condition} if there
exists a bounded nonempty open set $V$ such that
\[
\bigcup_{i=1}^{m}T_i(V)\subseteq V
\]
and the sets $T_i(V)$ are pairwise disjoint. Due to Falconer's result \cite{Falconer1995}, we have the following.

\begin{proposition}\label{prop:no-osc}
The iterated function system $\{S_0,S_A,S_B\}$ does not satisfy the
open set condition.
\end{proposition}

\begin{remark}
In an upcoming work,  we will present a very different construction in  $\mathbb{R}^{2}$ , based on the examples in \cite{BFM} , such that  the sub-self-similar set has positive Hausdorff dimension and the IFS satisfies the exponential separation condition (see for example \cite{HochmanRd}) .
\end{remark}

The author acknowledges the use of AI in the mathematical exploration and language writing.

\end{document}